\documentclass[10pt,reqno]{amsart}

\usepackage[english]{babel}
\usepackage{amsthm}
\usepackage{amsmath}
\usepackage{amssymb}
\usepackage{amstext}
\usepackage{latexsym}
\usepackage{enumitem}
\usepackage{mathrsfs}
\usepackage{stmaryrd}
\usepackage{comment}
\usepackage{hyperref}
\usepackage{url}
\usepackage[left=4cm, right=4cm,marginparwidth=3cm, marginparsep=0.25cm]{geometry}

\newtheorem{theorem}{Theorem}
\newtheorem{lemma}{Lemma}

\newtheorem{corollary}{Corollary}

\renewcommand{\ss}{\mathfrak{s}}
\renewcommand{\tt}{\mathfrak{t}}
\providecommand{\uu}{\mathfrak{u}}
\providecommand{\xx}{\mathfrak{x}}

\providecommand{\AAc}{\mathcal{A}}

\providecommand{\BBc}{\mathscr{B}}
\providecommand{\DD}{\mathsf{D}}
\providecommand{\FFc}{\mathscr{F}}
\providecommand{\GG}{G}
\providecommand{\NN}{\mathbb{N}}
\providecommand{\NNb}{\mathbb{N}}

\providecommand{\RR}{\mathbb{R}}
\providecommand{\ZZ}{\mathbb{Z}}
\providecommand{\ZZb}{\mathbb{Z}}

\providecommand{\diam}{{\rm diam}}
\providecommand\llb{\llbracket}
\providecommand\rrb{\rrbracket}

\hypersetup{
    pdftitle={Sums of dilates in ordered groups},
    pdfauthor={Alain Plagne and Salvatore Tringali},
    pdfmenubar=false,
    pdffitwindow=true,
    pdfstartview=FitH,
    colorlinks=true,
    linkcolor=red,
    citecolor=green,
    urlcolor=cyan
}

\begin{document}
\title{The Davenport constant of a box}
\author{Alain Plagne}
\address{Centre de math\'ematiques Laurent Schwartz, \'Ecole polytechnique, 91128 Palaiseau cedex, France}
\email{plagne@math.polytechnique.fr}

\author{Salvatore Tringali}
\address{Science Program, Texas A\&M University at Qatar, Education City \\ PO Box 23874 Doha, Qatar}
\email{salvo.tringali@gmail.com}
\urladdr{http://www.math.polytechnique.fr/~tringali/}

\thanks{This research is supported by the ANR Project CAESAR No. ANR-12-BS01-0011}

\subjclass[2010]{Primary: 11B75; Secondary: 11B30, 11P70}
\keywords{Additive combinatorics, Davenport constant, inverse theorem, minimal zero-sum sequence}

\begin{abstract}
Given an additively written abelian group $G$ and a set $X\subseteq G$, we let $\BBc (X)$
denote the monoid of zero-sum sequences over $X$ and $\DD(X)$ the Davenport constant
of $\BBc (X)$, namely the supremum of the positive integers $n$ for which there exists
a sequence $x_1 \cdots x_n$ of $\BBc (X)$ such that $\sum_{i \in I} x_i \ne 0$ for each
non-empty proper subset $I$ of $\{1, \ldots, n\}$. In this paper, we mainly investigate the case
when $G$ is a power of $\ZZb$ and $X$ is a box (i.e., a product of intervals of $G$).
Some mixed sets (e.g., the product of a group by a box) are studied too, and some inverse results
are obtained.
\end{abstract}
\maketitle

\section{Introduction}

Let $G$ be an additively written abelian group. Given $X \subseteq G$, we denote by $\FFc(X)$
the free abelian monoid of $G$ over $X$ and write it multiplicatively. Therefore, the reader should
be warned that $x^a$ is meant in this article as the sequence where $x$ is repeated $a$ times;
there will be no risk of confusion. We use $\BBc (X)$ for the abelian submonoid of $\FFc(X)$ of {\em zero-sum sequences over $X$},
that is containing all the non-empty words $x_1 \cdots x_n$ such that
$x_i \in X$ for each index $i$ and $\sum_{i=1}^n x_i = 0$, cf. \cite[Definition 3.4.1]{Ge-HK06a}.
Note that the sequences considered here are unordered.

Let $\ss = x_1 \cdots x_n$ be a non-empty sequence of $\BBc(X)$. By abuse of notation, we shall say that
the $x_i$'s are elements of $\ss$ or, simply, are in $\ss$ (that is, we identify sequences and multisets).
We say that $\ss$ is \textit{minimal} if $\sum_{i \in I} x_i \ne 0$ for every non-empty proper subset $I$
of $\{1, \ldots, n\}$. We call $n$ the {\em length} of $\ss$, which we denote by $\|\ss\|$, and we use $\AAc(X)$ for the set of minimal zero-sum sequences of $\BBc(X)$; notice that $\AAc (X) = \AAc (G) \cap \BBc ( X)$.
For further notation and terminology, we refer the reader to \cite[Section 2]{GG06}.

For $G$ an abelian group, the study of $\BBc(G)$ and its combinatorial properties is part of
what is called zero-sum theory, a subfield of additive theory with applications to group theory, graph theory, Ramsey theory,
geometry and factorization theory, see the survey \cite{GG06} and references therein.
One of the earliest questions in this area, and maybe one of the
most important, is concerned with the \emph{Davenport constant},
named after the mathematician who, according to \cite{Ol}, popularized it during the 1960s, starting
from a problem of factorization in algebraic number theory, see for instance  \cite{GB} or \cite{Ge-HK06a};
notice however that this group invariant was already discussed in \cite{Ro}.
The Davenport constant has become the prototype of algebraic invariants
of combinatorial flavour. Since the 1960s, the theory of these invariants
has highly developed in several directions, see for instance the survey article
\cite{GG06} or \cite[Chapters 5, 6, and 7]{Ge-HK06a}.

Given a finite abelian group $G$, it turns out that any long enough sequence of elements
in it contains a zero-sum subsequence. More generally, the Davenport constant of an abelian group $G$,
denoted by $\DD (G)$, is defined as the smallest integer $n$ such that each sequence over $G$ of length
at least $n$ has a non-empty zero-sum subsequence.
Equivalently, $\DD (G)$ is the maximal length of a minimal zero-sum sequence over $G$, \emph{i.e.}
the maximal length of a sequence of elements of $G$ summing to $0$ and with no proper subsequence
summing to $0$. If $G$ is decomposed, as is always possible if $G \neq \{ 0 \}$, as a direct sum
of cyclic groups $G\cong C_{n_1} \oplus \cdots \oplus C_{n_r}$ with integers
$1< n_1 \mid \dots \mid n_r$ (here, $C_k$ denotes a cyclic group with $k$ elements,
$r$ is the rank of $G$ and $n_r = \exp G$, the exponent of $G$),
an immediate lower bound for the Davenport constant is
\begin{equation}
\label{lowb}
\DD (G) \geq 1+ \sum_{i=1}^r (n_i - 1);
\end{equation}
to see this, notice that the sequence containing, for each $i = 1, \ldots, r$, one generator of the cyclic
component $C_{n_i}$ repeated $n_i-1$ times,
has no non-empty zero-sum subsequence. It is known that for groups of rank at most two
and for $p$-groups (with $p$ a prime), \eqref{lowb} is in fact an equality, as
was obtained independently in \cite{vEB} and \cite{Ol,Ol2}. In particular, if $G$ is cyclic then
\begin{equation}
\label{cyclic}
\DD (G) = |G |,
\end{equation}
and this is characteristic of cyclic groups, as for instance follows immediately from \eqref{vebkmesh}.
For groups of rank at least four, equality is definitely not
the rule, see \cite{Baa, vEB, GS}. In the case of groups of rank three, it has been conjectured that equality holds again, but this conjecture is wide open, see \cite{GG06}, and
seemingly difficult. Concerning upper bounds, the best general result is the following:
\begin{equation}
\label{vebkmesh}
\DD (G) \leq  \left( 1 + \log \frac{|G|}{\exp G} \right) \exp G,
\end{equation}
which is proved in \cite{vEBK, Meshu}. We do not know really more than this in general: In spite of so much work related
to the Davenport constant over the years, its actual value has been determined only for a few additional
families of groups beyond the ones for which it was already known by the end of the 1960s.
The general impression is that, although it has a very simple definition,
computing the Davenport constant of an abelian group (of rank at least three) is a challenging problem.

Let it be as it may, it turns out that generalizing the question to a broader setting makes sense and can be useful.
In particular, for any subset $X$ of an abelian group $G$ we may define
its Davenport constant, which we denote by $\DD(X)$, as the largest integer $n$ for which
there exists a minimal zero-sum sequence in $\BBc (X)$ of length $n$; this variant was first introduced by van Emde Boas in \cite{vEB}, where it is however denote by $\mu(G, X)$.
It is trivial but worth remarking that in general, and contrarily
to the case where $X = G$, it can happen that $\DD(X)$ is finite and yet we can build arbitrarily long sequences with no non-empty proper zero-sum subsequence.
Also, it is immediate that $\DD( X) \leq \DD( G)$: this inequality is in general strict and it is well possible that $\DD( X)$ is finite while $\DD( G)$ is not.

The study of such a generalisation of the Davenport constant to subsets of abelian groups, is
of great interest for its applications to factorization theory, an area which is currently expanding from the classical setting of (mostly commutative) rings to the context of modules.
Indeed, if $H$ is a Krull monoid with class group $G$ and  if $X \subseteq G$ is the set of classes
containing prime divisors, then the Davenport constant $\DD (X)$ is a crucial invariant describing
the arithmetic of $H$, see \cite[Chapter 3.4]{Ge-HK06a} and \cite{GGSS}. It turns out that
the study of direct-sum decompositions in module theory gives rise to Krull monoids with class groups
which are precisely a power of the additive group $\ZZ$ of the integers. For this reason, Baeth and Geroldinger, in the final section of their recent paper \cite{BaethG}, ask specifically, as part of a larger research programme, to study
the Davenport constant of what we call a box, that is a product of intervals of integers.

The main goal of the present paper is, in fact, to derive bounds and exact formulas for $\DD(X)$
in the case when $X$ is a subset of a power of the additive group $\ZZ$ of the integers; in particular, we mostly investigate the case of $X$ being a box. Inverse results, describing the structure of the
sequences of maximal or almost maximal length, are also presented, along with hybrid
results involving the product of a group and a box.

\section{New results}

The first part of our study is concerned with the case of the integers; interesting results in this direction have been recently obtained by Sissokho in \cite{Sisso}.
As usual, we let the {\em diameter} of a set $X \subseteq \ZZ$ be given by
$$
\diam(X) = \sup_{x,y \in X} |x-y|
$$
and we denote, in all what follows, by $\chi$ the function defined, for all subsets
$S$ of $\ZZb$ containing both positive and negative elements, by the formula
$$
\chi (S) = \sup_{x,y \in S \text{ with } xy < 0}\quad \frac{|x|+|y|}{\gcd(x,y)}.
$$

Our first result can be then stated as follows.

\begin{theorem}
\label{th:main_1}
Let $X$ be a non-empty set of integers. Then,
\begin{enumerate}[label={\rm (\roman{*})}]
\item if $X \subseteq \NNb \setminus \{ 0 \}$ then $\DD(X)= 0$,
\item if $0 \in X \subseteq \NNb$ then $\DD(X)=1$,
\item if $X$ contains both positive and negative integers, then $\chi(X) \le \DD(X) \le  \diam(X)$.
\end{enumerate}
\end{theorem}

Since there are sets $X$ for which $\chi(X) = \diam(X)$ (consider, e.g., the interval
$\llb -m,M \rrb$, where $m$ and $M$ are coprime positive integers, or apply Corollary \ref{cor:2a}),
point (iii) is in general sharp. We recall that, if $a$ and $b$ are real numbers, $a\leq b$, by $[a, b]$
we denote the interval $\{x \in \RR \text{ such that } a \le x \le b\}$, while we write $\llb a, b\rrb$ for the set $[a,b] \cap \ZZ$.

On the other hand, as will follow from our forthcoming results, there are sets $X$ such
that $\DD(X) <  \diam(X)$ (see, for instance, Corollary \ref{cor:2}).
Yet, we do not know of a single example for which $\chi(X) < \DD(X)$.
However, we have the following corollary (immediate from Theorem \ref{th:main_1})
in the case that $X$ is an interval around zero.

\begin{corollary}
\label{cor:2a}
Let $m$ and $M$ be positive integers, we have
$$
\frac{m+M}{\gcd(m,M)} \le \DD(\llb -m, M \rrb) \le m + M.
$$
In particular, if $m$ and $M$ are coprime, then
$$
\DD(\llb -m, M \rrb) = m + M.
$$
\end{corollary}

From this first corollary, one can immediately deduce the value of the Davenport
constant of a symmetrical interval around zero.

\begin{corollary}
\label{cor:2}
We have $\DD(\llb -1,1\rrb) = 2$ and, for any integer $m \geq 2$, $\DD(\llb -m,m \rrb) = 2m-1$.
\end{corollary}

Moreover, the following asymptotic estimate holds.

\begin{corollary}
\label{cor:bertrand}
For positive integers $m$ and $M$, one has:
$$
\DD(\llb -m, M \rrb) = M + m + o\big( \min (m,M) \big)\qquad {\rm as}\ \min (m,M ) \to +\infty.
$$
\end{corollary}

It will be transparent from the proof that, in Corollary \ref{cor:bertrand}, we can replace the error term
$o\big( \min (m,M) \big)$ with an explicit power (slightly larger than $1/2$) of $\min (m,M)$.

In fact, Corollary \ref{cor:2} appears (in an alternate but equivalent form) as part of the main theorem
in \cite{Sahs}, where the focus is mainly on pairs $(A,B)$ of non-empty subsets of
\textit{positive} integers, therein referred to as irreducible pairs, such that
$\sum_{a \in A} a = \sum_{b \in B} b$ and $\sum_{a \in A^\prime} a \ne \sum_{b \in B^\prime} b$
for any other pair $(A^\prime, B^\prime)$ of non-empty sets $A^\prime \subsetneq A$ and
$ B^\prime \subsetneq B$.

In the present paper, we shall adopt a strategy which looks quite different, both in spirit and
in practice. In particular, the proof of Corollary \ref{cor:2} comes very quickly as a consequence
of a technical lemma (essentially, Lemma \ref{lem:pigeon} \ref{pig1} of
Section \ref{sec:preliminary_lemmas}) of general interest and which we reuse to go a step further.

Having a direct theorem at hand, we are naturally led to its inverse counterpart. The first result
we obtain in this direction is concerned with the structure of minimal zero-sum sequences
of maximal length in an interval.

\begin{theorem}
\label{th:main_2}
Let $m$ and $M$ be positive integers and let $\ss = x_1 \cdots x_{m+M}$ be a sequence
of length $m+M$ in $\BBc(\llb -m, M \rrb)$.
Then, $\ss$ is minimal if and only if $\gcd(m,M) = 1$ and $\ss = M^{m} \cdot (-m)^M$.
\end{theorem}

This in turn leads to the following corollary.

\begin{corollary}
\label{cor:3}
Let $m \ge 2$ be an integer and let $\ss = x_1 \cdots x_{2m-1}$ be a sequence of length $2m-1$ in $\BBc(\llb -m, m \rrb)$.
Then, $\ss$ is minimal if and only if $\ss = m^{m-1} \cdot (-(m-1))^m$ or $\ss = (-m)^{m-1} \cdot (m-1)^m$.
\end{corollary}

Our next theorem is a more elaborate inverse result which reads as follows.

\begin{theorem}
\label{th:inverse_2m-2}
Let $m$ be an integer, $m \ge 3$, and let $\ss = x_1 \cdots x_{2m-2}$ be a sequence of length $2m-2$
in $\BBc(\llb -m, m \rrb)$. Then, $\ss$ is minimal if and only if one of the following holds:
\begin{enumerate}[label={\rm (\roman{*})}]
\item $m$ is odd and either $\ss = m^{m-2} \cdot (-m+2)^m$ or $\ss = (-m)^{m-2} \cdot (m-2)^m$;
\item $\ss = m^{m-2} \cdot (-(m-1))^{m-1} \cdot 1$ or $\ss = (-m)^{m-2} \cdot (m-1)^{m-1} \cdot (-1)$.
\end{enumerate}
\end{theorem}

The next theorem is a partial generalisation of the upper bound in Corollary \ref{cor:2} to higher dimensions.
It will follow from the connection, already noticed in \cite{DGS}, of the Davenport constant with the Steinitz
constant \cite{Stein13} and a generalisation of it obtained in \cite{Banasz91}.

\begin{theorem}
\label{barbie}
Let $m_1, \ldots, m_d$ be positive integers, we have
$$
\DD( \llb -m_1, m_1 \rrb \times \cdots \times \llb -m_d, m_d \rrb  ) \leq  \prod_{i=1}^d  \left(  2 \left( d+ \frac1d - 1 \right) m_i +1\right).
$$
\end{theorem}

Our next result is concerned with the special case of hypercubes.
We shall need a Kronecker-type notation (defined on positive integers $m$), namely
$$
\delta_m =
\left\{
\begin{array}{ll}
1 & \text {if } m=1,\\
0 & \text{otherwise.}
\end{array}
\right.
$$
We obtain the following bounds.

\begin{theorem}
\label{th:main_3}
One has
\begin{enumerate}[label={\rm (\roman{*})}]
\item $ \DD(\llb -1,1\rrb^2)=4$,
\item for any integer $m \geq 2$,
$$
(2m-1)^2 \le \DD(\llb -m,m\rrb^2) \le (2m+1)(4m+1),
$$
\item if $d$ is an integer, $d\ge 3$, and $m$ is positive integer,
$$
(2m-1+ \delta_m)^d \le \DD(\llb -m,m\rrb^d) \le \left( 2 \left(d+\frac1d -1 \right) m+1 \right)^d.
$$
\end{enumerate}
\end{theorem}

The lower bounds in this theorem are obtained thanks to direct constructions, while
the upper bounds follow immediately from Theorem \ref{barbie}.
Theorem \ref{th:main_3} being proved, the general impression,
supported by the special cases of the dimension $d=1$
and the square $\llb -1,1\rrb^2$, is that the true size of $\DD(\llb -m,m\rrb^d)$
is closer to the lower bound than to the upper bound.

We notice that in \cite{XXXX} the authors consider the case
$$
X = \llb 0, 1 \rrb^d \cup \llb -1,0 \rrb^d  \setminus \{ 0^d \}
$$
where $0^d$ is the origin in $\RR^d$, and they prove a result that is reminiscent of our
Theorem \ref{th:main_3} (iii), see \cite[Theorem 3.13]{XXXX}. Loosely speaking, they obtain the bounds
\begin{equation}
\label{gerol}
\left( \frac{1+ \sqrt{5}}{2} \right)^d  \leq \DD (X) \leq (d+2)^{(d+2)/2}.
\end{equation}
Although this set $X$ is not an hypercube, as we consider here, we may still force
the (somewhat unnatural) direct application of the upper bound of Theorem \ref{th:main_3}
(our lower bound gives nothing in this case), which implies for this case that
$\DD (X) \leq \DD ( \llb -1,1\rrb^d) \le \left( 2d+ 2/d -1 \right)^d$ and is definitely worse than
\eqref{gerol}, but still of the same ``type''. It would be interesting to check if our method could be efficiently
adapted  to this special case.

We notice that Theorem \ref{th:main_3} is enough to ensure that, for fixed $d$,
the quantity $\DD(\llb -m,m\rrb^d)$ grows like $m^d$.
But it is not clear that a constant $a_d$ should exist so that
$$
\DD(\llb -m,m\rrb^d) \sim a_d m^d\quad \text{ as } m \to + \infty.
$$
However, if such a constant exist it must satisfy $2^d \leq a_d \leq (2 \left(d+1/d -1 \right))^d$.

Based on the above, we are led to ask whether, $m$ and $d$ being given as in the statement
of Theorem \ref{th:main_3}, the Davenport constant of the hypercube $\llb -m, m \rrb^d$
is equal to the $d$-th power of the Davenport constant of $\llb -m, m \rrb$.
Should this be true, it would suggest that some suitable assumptions could imply a sort
of multiplicativity of Davenport constants for certain classes of sets.
Our two last theorems and their corollary go more generally in this direction.
The first of these theorems is a submultiplicativity result.

\begin{theorem}
\label{th:main_4}
Let $G$ and $H$ be two abelian groups. If $G$ is finite and $X$ is a finite subset of $H$,
then
$$
\DD(G \times X) \leq   \DD(G)\ \DD (X).
$$
\end{theorem}

The final theorem shows a supermultiplicativity property, not with respect to the Davenport
constants themselves but rather with respect to the lower bounds offered by Theorem \ref{th:main_3}.
Indeed, we shall build long minimal zero-sum sequences on the basis of those already built
for each component.

\begin{theorem}
\label{th5}
Let $m$ and $d$ be positive integers and let $G$ be a cyclic group, then
$$
\DD(G \times \llb -m, m \rrb^d ) \geq  \DD(G) (2m-1+ \delta_m)^d.
$$
\end{theorem}

In general, both theorems are sharp, as shown by our final corollary which follows in an immediate way
from Theorems \ref{th:main_4} and \ref{th5} and Corollary \ref{cor:2}.

\begin{corollary}
\label{cor:4}
Let $m$ be a positive integer and let $G$ be a cyclic group, then
$$
\DD(G \times \llb -m, m \rrb ) =  \DD(G)\ \DD (\llb -m, m \rrb).
$$
\end{corollary}

The plan of the paper is as follows. In Section \ref{sec:preliminary_lemmas}, we establish a few lemmas
of general interest and which will be useful in the other parts of the article.
In Section \ref{secThm1}, we prove Theorem \ref{th:main_1} and Corollaries \ref{cor:2} and \ref{cor:bertrand}.
Section \ref{section-inverse} contains the proofs of the inverse results, namely Theorem \ref{th:main_2} and
its Corollary \ref{cor:3} and of Theorem \ref{th:inverse_2m-2}. Finally the proofs of Theorems \ref{barbie} and \ref{th:main_3}
are presented in Section \ref{avtder}, while Section \ref{der} contains the proofs of our final Theorems \ref{th:main_4} and \ref{th5}.

\section{Preliminary lemmas}
\label{sec:preliminary_lemmas}

In this section, we collect a few lemmas that will be used later to prove our main results.
We start with the following elementary lemma, the proof of which is immediate (and hence omitted).

\begin{lemma}
\label{lem:trivial}
Let $\ss = x_1 \cdots x_n$ be a non-empty minimal zero-sum sequence of an abelian group $\GG$. Then, we have:
\begin{enumerate}[label={\rm (\roman{*})}]
\item\label{lem:trivial:item_0} the sequence $- \ss = (-x_1) \cdots (-x_n)$ is itself
a non-empty minimal zero-sum sequence of $\GG$,
\item\label{lem:trivial:item_ii} $0 \in \ss$ if and only if $n = 1$,
\item\label{lem:trivial:item_iii} the elements $x$ and $-x$ are both in $\ss$ for some $x \in G\setminus \{0\}$
if and only if $n = 2$.
\end{enumerate}
\end{lemma}

The next lemma gives some elementary properties of the function $\DD$.
It turns out that it is an even and non-decreasing function.
As is usual, we shall denote
$$
-X = \{-x \text{ for } x \in X\}.
$$

\begin{lemma}
\label{lem:monotony}
Let $G$ be an abelian group.
If $X \subseteq Y \subseteq G$, then
$\DD(X) \le \DD(Y)$ and $\DD(-X) = \DD(X)$.
Moreover, $\DD(\llb -m, m \rrb) = \DD( \llb -(m-1), m \rrb)$ for every integer $m \ge 2$.
\end{lemma}

\begin{proof}
The first inequality is immediate. The second one follows from Lemma \ref{lem:trivial} \ref{lem:trivial:item_0}.
As for the third one, the first inequality implies $\DD(\llb -m, m \rrb) \geq  \DD( \llb -(m-1), m \rrb)$.
Now, we notice that for $m \ge 2$, $\DD(\llb -m, m \rrb)>2$
since the sequence $-m \cdot 1^m \in \AAc (\llb -m, m \rrb)$ has length $m+1 \geq 3$.
It follows by Lemma \ref{lem:trivial} \ref{lem:trivial:item_iii} that a minimal zero-sum sequence
of  $\llb -m, m \rrb$ cannot contain both $m$ and $-m$ and, therefore,
up to symmetry, is included in $\llb -m+1, m \rrb$. This proves the third assertion of the lemma.
\end{proof}

Our methods heavily rely on considering partial sums of terms of the sequences we study.
The following lemma is the first result of a series in this direction.

\begin{lemma}
\label{lem:all_distinct}
Let $\ss = x_1 \cdots x_n$ be a non-empty minimal zero-sum sequence of an abelian group $\GG$.
Then, for any permutation $\sigma$ of $\llb 1, n \rrb$ and all $i,j \in \llb 1, n \rrb$, the following holds:
$\sum_{l=1}^i x_{\sigma(l)} \ne \sum_{l=1}^j x_{\sigma(l)}$ if and only if $i \ne j$.
\end{lemma}

\begin{proof}
Suppose the result is false: there exist a permutation $\sigma$ of $\llb 1, n \rrb$ and distinct indices $i,j \in \llb 1, n \rrb$
such that $\sum_{l=1}^i x_{\sigma(l)} = \sum_{l=1}^j x_{\sigma(l)}$. By symmetry, we can assume $i < j$.
This yields that the non-empty sum $\sum_{l=i+1}^j x_{\sigma(l)} = 0$ that is, $x_{\sigma(i+1)} \cdots x_{\sigma(j)}$
is a proper non-empty zero-sum subsequence of $\ss$, which is impossible by the minimality of $\ss$.
\end{proof}

Here is a useful companion result to the preceding lemma.

\begin{lemma}
\label{refine}
Let $\ss = x_1 \cdots x_n$ be a non-empty minimal zero-sum sequence of length $n \geq 3$ of an abelian group $\GG$.
Then, for any permutation $\sigma$ of $\llb 1, n \rrb$ and any index $i \in \llb 1, n \rrb \setminus  \{ 2 \}$,
the value of $\sum_{l=1}^i x_{\sigma(l)}$ is different from $x_{\sigma(1)}+x_{\sigma(3)}$.
\end{lemma}

\begin{proof}
If $x_{\sigma(1)} = x_{\sigma(1)}+x_{\sigma(3)}$, then $x_{\sigma(3)}=0$, a contradiction
by Lemma \ref{lem:trivial} \ref{lem:trivial:item_ii} since $n>1$: this solves the case $i=1$; while, if for some $i \geq 3$,
$$
\sum_{l=1}^i x_{\sigma(l)} = x_{\sigma(1)}+x_{\sigma(3)}
$$
then
$$
x_{\sigma(2)}+\sum_{l=4}^i x_{\sigma(l)}=0
$$
(if $i=3$, the sum on $l$ on the left-hand side is empty) which contradicts the minimality of $\ss$.
\end{proof}

The two preceding lemmas will be used under the form of the following counting lemma
which will be key in several proofs.

\begin{lemma}
\label{lem:pigeon}
Let $\ss = x_1 \cdots x_n$ be a non-empty minimal zero-sum sequence of an abelian group $\GG$.
We assume that there exist a set $X$ and a permutation $\sigma$ of $\llb 1,n \rrb$ such that
for any $i \in \llb 1, n \rrb$, the partial sum $\sum_{l=1}^i x_{\sigma(l)}$ belongs to $X$.
Then
\begin{enumerate}[label={\rm (\roman{*})}]
\item\label{pig1} the inequality $n \le |X|$ holds true,
\item\label{pig2} if we assume additionally that $n \ge 3$, $x_{\sigma (2)} \neq x_{\sigma (3)}$
and $x_{\sigma (2)} + x_{\sigma (3)} \in X$, then $n \le |X|-1$.
\end{enumerate}
\end{lemma}

\begin{proof}
By Lemma \ref{lem:all_distinct}, all the partial sums $\sum_{l=1}^i x_{\sigma(l)}$ ($1 \le i \le n$)
must be pairwise distinct. Since, by assumption, all these elements belong to $X$, this implies $n \le |X|$.

If $n \ge 3$, we may additionally apply Lemma \ref{refine}. Since, by assumption,
$x_{\sigma (2)} \neq x_{\sigma (3)}$, we obtain that for $i \in \llb 1, n \rrb$,
the partial sums $\sum_{l=1}^i x_{\sigma(l)}$  are pairwise distinct
and different from $x_{\sigma (1)} + x_{\sigma (3)}$. We obtain
$$
\left| \left\{  \sum_{l=1}^i x_{\sigma(l)}, \text{ for } 1 \leq i \leq n    \right\} \cup \{ x_{\sigma (1)} + x_{\sigma (3)} \} \right|=n+1.
$$
Since all the $n+1$ elements appearing in the left-hand side of this equality are in $X$, the result follows.
\end{proof}

A classical consequence of Lemma \ref{lem:pigeon} is the well-known fact that if $\GG$ is
a finite abelian group, then $\DD (G) \leq | G |$ (this bound is sharp, as is seen in \eqref{cyclic}).

Now we introduce a technical definition. We shall say that a triple $(\ss, k, \sigma)$ is {\em nyctalopic}
if $\ss = x_1 \cdots x_n$ is a minimal zero-sum sequence of $\BBc(\ZZ)$ of length $n \ge 2$,
$k$ is an integer in the range $1 \leq k \leq n$ and $\sigma$ is an injective function defined on
$\llb 1, k \rrb$ and taking its values in $\llb 1, n \rrb$ such that
the following property holds: for any $i \in \llb 2, k  \rrb$, one has
$$
x_{\sigma(i)}  \sum_{l=1}^{i-1} x_{\sigma(l)}< 0.
$$
When $k=n$, if there is no risk of confusion (that is, if which $\ss$ is involved is clear from the context),
we will simply say that $\sigma$ is a nyctalopic permutation.

Nyctalopic triples $(\ss, k, \sigma)$ have nice properties which justify their introduction.
The following lemma of an algorithmic nature will be very useful in what follows.

\begin{lemma}
\label{prolong}
Let $X$ be a finite subset of $\ZZ$.
Let $\ss = x_1 \cdots x_n$ be a minimal zero-sum sequence of $\BBc(X)$ of length $n \ge 2$.
Let $k$ be an integer, $1\leq k \leq n$, and $\sigma$ be an injective function defined on $\llb 1,k \rrb$ and
taking its values in $\llb 1, n \rrb$ such that the triple $(\ss, k, \sigma)$ is nyctalopic.
Then, one can extend $\sigma$ to a nyctalopic permutation of $\llb 1, n \rrb$.
\end{lemma}

\begin{proof}
We proceed by induction. By assumption, $(\ss, k, \sigma)$ is nyctalopic.

Assume now that for some integer $i \in \llb k, n-1 \rrb$, $\sigma$ has been extended so that the values of $\sigma (k+1),\dots, \sigma (i)$ are determined
in such a way that $(\ss, i, \sigma)$ is nyctalopic. It is immediate to check that
$$
 \sum_{l=1}^{i} x_{\sigma(l)} \neq 0
$$
since otherwise $\ss$ would not be a minimal zero-sum sequence in view of $i<n$.
Since $\ss$ sums to zero there should be at least one integer
$j \not\in \{ \sigma (l) \text{ for } 1 \leq l \leq i\}$ such that $x_j$ has a  sign opposite
to the one of $\sum_{l=1}^{i} x_{\sigma(l)}$. We fix one of these integers $j$ arbitrarily.
Then we extend $\sigma$ by defining
$$
\sigma (i+1) =j
$$
so that, by construction, $(\ss, i+1, \sigma)$ is nyctalopic.
\end{proof}

Here is the central property of nyctalopic triples we use in what follows.

\begin{lemma}
\label{prop-nyctalopic}
Let $\ss$ be a minimal zero-sum sequence of $\BBc(X)$ of length $n \ge 2$.
Let $\sigma$ be a nyctalopic permutation of $\llb 1, n \rrb$.
Then, for any $i \in \llb 1, n \rrb$,
$$
\min X \leq  \sum_{l=1}^{i} x_{\sigma(l)} \leq \max X.
$$
Moreover, if $x_{\sigma (1)} \neq \max X$, the inequality on the right is strict
while, if  $x_{\sigma (1)} \neq \min X$, the inequality on the left is strict.
\end{lemma}

\begin{proof}
Notice first that $n \geq 2$ implies $\min X < 0 < \max X$, as follows from Theorem \ref{th:main_1} (i) and (ii).

The assertion of Lemma \ref{prop-nyctalopic} is proved by induction, the lemma being trivial for $i=1$.
Suppose it is true for some $i \in \llb 1, n-1 \rrb$, we thus have
$$
\min X \leq  \sum_{l=1}^{i} x_{\sigma(l)} \leq \max X.
$$
By minimality, this sum is also non-zero since $i<n$.
Suppose that $\sum_{l=1}^{i} x_{\sigma(l)} >0$ then by nyctalopia,
one has $x_{\sigma (i+1)} <0$ that is, $\min X \leq x_{\sigma (i+1)}  \leq -1$
and thus
$$
 \min X < 1 +\min X \leq  \sum_{l=1}^{i} x_{\sigma(l)} +x_{\sigma (i+1)} \leq \max X -1 < \max X.
$$
The case $\sum_{l=1}^{i} x_{\sigma(l)} <0$ is treated in a symmetric way.
\end{proof}

\section{Proof of Theorem  \ref{th:main_1} and its corollaries}
\label{secThm1}

We start with a lemma.

\begin{lemma}
\label{seq}
Let $x$ and $y$ be integers such that $xy < 0$ and let $X= \{ x,y \}$. Then
\begin{enumerate}[label={\rm (\roman{*})}]
\item\label{seq1} the set $\AAc ( X)$ has a unique element, $\xx =x^{a}\cdot y^{b}$
with $a =|y|/\gcd(x,y)$ and $b = |x|/\gcd(x,y)$,
\item\label{seq2} the set $\BBc (X)$ is equal to $\{ \xx^j \text{ for } j \in \NN \}$.
\end{enumerate}
\end{lemma}

\begin{proof}
By definition, the sequence $x^{a}\cdot y^{b}$ is in $\BBc(X)$ if and only if $ax + by = 0$, that is $a|x| = b|y|$.
The preceding equality can be rewritten as
$$
a \frac{|x|}{\gcd(x,y)} = b \frac{|y|}{\gcd(x,y)}.
$$
But $ |x|/\gcd(x,y)$ and $|y|/\gcd(x,y)$ are coprime, therefore Gauss lemma gives the
existence of a non-negative integer $h$ such that $b = h |x|/\gcd(x,y)$ and $a = h |y|/\gcd(x,y)$.
This proves \ref{seq2}.

Among these sequences, only the one corresponding to $h=1$ is minimal (and divides those for $h \geq 1$)
and \ref{seq1} follows.
\end{proof}

Here is the very proof of the Theorem.

\begin{proof}[Proof of Theorem \ref{th:main_1}]
The points (i) and (ii) are immediate. We thus turn directly to (iii).

In order to prove $\chi (X) \leq \DD (X)$, we consider, for all $x,y \in X$ with $xy < 0$,
the sequence $\ss=x^a \cdot y^b$ where
$$
a=\frac{|y|}{\gcd(x,y)}\qquad \text{ and }\qquad  b=\frac{|x|}{\gcd(x,y)}.
$$
By Lemma \ref{seq} \ref{seq1}, this is a minimal zero-sum sequence.
Consequently, we obtain
$$
\frac{|x| + |y|}{\gcd(x,y)} = || \ss || \leq \DD (G),
$$
hence the result, on taking the supremum on the left-hand side.

On another hand, the upper bound $\DD (X) \le \diam(X)$ is trivial if $|X| = + \infty$.
So assume that $X$ is finite, and let $m = -\min X$ and $M = \max X$.
If $\ss = x_1 \cdots x_n \in \AAc (X)$, then $\|\ss\| \ge \chi(X) \ge 2$ by the inequality
we just proved. Define $\sigma (1)=1$. Lemma \ref{prolong} implies
that we can extend $\sigma$ into a nyctalopic permutation of $\llb 1, n \rrb$.
Lemma \ref{prop-nyctalopic} then implies that all the partial sums $x_{\sigma(1)} + \cdots + x_{\sigma(i)}$
(for $1 \leq i \leq n$) belong to either $\llb -m, M-1 \rrb$ or $\llb -(m-1), M \rrb$,
with the result that $n \le M+m = \chi(X)$, in view of Lemma \ref{lem:pigeon} \ref{pig1}.
\end{proof}

We conclude the section with the proof of the two corollaries to Theorem \ref{th:main_1}.

\begin{proof}[Proof of Corollary \ref{cor:2}]
By Corollary \ref{cor:2a}, the claim is trivial if $m = 1$,
while Lemma \ref{lem:monotony} and Corollary \ref{cor:2a} give
$$
\DD (\llb -m, m \rrb)= \DD (\llb -(m-1), m \rrb) = 2m-1
$$
for $m \ge 2$ since in this case $\gcd (m-1,m)=1$.
\end{proof}

For the proof of Corollary \ref{cor:bertrand}, we shall need the symbol $[x]$ for the integral part
(by default) of a real number $x$.

\begin{proof}[Proof of Corollary \ref{cor:bertrand}]
Since Hoheisel \cite{Hoheisel}, we know that for some $\vartheta <1$, when $x$ is large enough,
there is always a prime $p_x$ in the real open interval  $(x - x^\vartheta, x )$.
One can even take $\vartheta=0.525$ \cite{BHP}.

Assume $\min (m,M)=m$ (the other case is analogous). Applying Hoheisel's result, we may find
a prime $p$ in $\llb  m - [m^\vartheta], m \rrb$. Since $p$ cannot divide $M$ and $M-1$ at
the same time, there must exist $\eta =0$ or 1 such that $\gcd (M- \eta , p)=1$.
We infer
$$
p+M-1 \leq p+M -\eta = \frac{p+ M-\eta }{\gcd(p,M- \eta )} \le \DD(\llb -p, M-\eta \rrb) \le  \DD(\llb -m, M \rrb) \le m + M,
$$
where we have used the coprimality of $p$ and $M- \eta$, Corollary \ref{cor:2a} and
the non-decreasingness of $\DD$ given by Lemma \ref{lem:monotony}. The result follows since
$$
p+M - 1 = m+M + O(m^{\vartheta }+1) = m+M+ o(m).
$$
\end{proof}

\section{Proofs of the inverse theorems and their corollaries}
\label{section-inverse}

We start with the proof of Theorem \ref{th:main_2}.

\begin{proof}[Proof of Theorem \ref{th:main_2}]
That the condition of the Theorem is sufficient follows from Lemma \ref{seq} \ref{seq1}. We now investigate
its necessity.

Suppose that $\ss$ contains an element $x_i$ different from both $-m$ and $M$. Define $\sigma (1)=i$ and
apply Lemma \ref{prolong} in order to extend $\sigma$ into a nyctalopic permutation.
By Lemma \ref{prop-nyctalopic}, we obtain that the partial sums $x_{\sigma(1)} + \cdots + x_{\sigma(j)}$ all
belong to $\llb -(m-1), M-1 \rrb$. This in turn implies $\|\ss\| \le M+m-1$ by Lemma \ref{lem:pigeon} \ref{pig1},
which is a contradiction.

It follows that $\ss$ is of the form $(-m)^a \cdot M^b$ for some positive integers $a$ and $b$, that is,
$\ss \in \BBc ( \{ -m,M \} )$.
By Lemma \ref{seq} \ref{seq1}, the minimality of $\ss$ implies
$$
a = \frac{M}{\gcd(M,m)}\quad \text{ and }\quad  b=\frac{m}{\gcd(M,m)}.
$$
From the assumption and this, we deduce that
$$
M+m= || \ss || =a+b =  \frac{M+m}{\gcd(M,m)}
$$
and $\gcd (M,m)=1$ follows.
\end{proof}

The proof of its corollary is now easy.

\begin{proof}[Proof of Corollary \ref{cor:3}]
By Lemma \ref{lem:trivial} \ref{lem:trivial:item_iii}, since $2m-1>2$, $\ss$ cannot contain both $m$ and $-m$.
Assume that $\ss$ does not contain $-m$, then it belongs to $\BBc(\llb -(m-1), m \rrb)$ and we apply Theorem \ref{th:main_2},
which gives the result. The case where $\ss$ does not contain $m$ is analogous.
\end{proof}

We now come to the second inverse result. It turns out that its proof is by far more intricate than the preceding one.

\begin{proof}[Proof of Theorem \ref{th:inverse_2m-2}]

In this proof, we will distinguish two cases (cases (i) and (ii)), the first one being very simple.
The second case will use two internal lemmas (Lemmas \ref{neg} and \ref{pos} below).

Since $\DD (\llb -(m-1),m-1 \rrb) = 2m-3$ by Corollary \ref{cor:2}, we can assume by symmetry
and point \ref{lem:trivial:item_iii} of Lemma \ref{lem:trivial} that $m \in \ss$ and $-m \notin \ss$.
In other words $\ss \in \BBc(\llb -(m-1), m \rrb)$.

We distinguish two cases, the first one being almost immediate.
\smallskip

(i) If $-(m-1) \notin \ss$, then $\ss \in \BBc(\llb -(m-2), m \rrb )$.
It follows from Theorem \ref{th:main_2} that $\ss$ is the sequence
$m^{m-2} \cdot (-(m-2))^m$ and $\gcd(m-2,m) = 1$, i.e. $m$ is odd.
\smallskip

(ii) If $-(m-1) \in \ss$, then point \ref{lem:trivial:item_iii} of Lemma \ref{lem:trivial} implies that $m-1 \notin \ss$.
Up to reordering the elements of $\ss$, we may therefore assume from now on that
$$
x_1 = m\qquad \text{ and }\qquad x_2=-(m-1).
$$

\begin{lemma}
\label{neg}
If, for some $i \in  \llb 3, n \rrb$, $x_i$ is negative then $x_i = -(m-1)$.
\end{lemma}

\begin{proof}
Suppose the lemma is false and let us consider an index $i \geq 3$ such that $-(m-1) < x_i \leq -1$.
We consider $\sigma$ the function defined on  $\llb 1, 3 \rrb$ by
$$
\sigma (1)=1, \sigma (2)=2, \sigma (3)=i.
$$
The triple $(\ss, 3, \sigma)$ is nyctalopic. We apply Lemma \ref{prolong} to $(\ss, 3, \sigma)$
to extend $\sigma$ into a nyctalopic permutation of  $\llb 1, n \rrb$. We then apply
Lemma \ref{prop-nyctalopic}. We infer that all the partial sums
$\sum_{l=1}^j x_{\sigma(l)}$ ($1 \leq j \leq n$) belong to $\llb -(m-2), m \rrb$.

But in fact even the following more precise statement is true, namely
\begin{equation}
\label{m-3}
\sum_{l=1}^j x_{\sigma(l)} \in  \llb -(m-3), m \rrb .
\end{equation}
This is the case when $j=1$ or $2$ and, indeed, if for some $j \geq 3$, one has $\sum_{l=1}^j x_{\sigma(l)}=-(m-2)$,
then by definition of nyctalopia the sum $\sum_{l=1}^{j-1} x_{\sigma(l)}$ is either $< -(m-2)$ or positive.
Since the first possibility is not possible (all the partial sums are at least $-(m-2)$) then
$\sum_{l=1}^{j-1} x_{\sigma(l)} \geq 1$. It follows that
$$
x_{\sigma(j)} = \sum_{l=1}^{j} x_{\sigma(l)}  -  \sum_{l=1}^{j-1} x_{\sigma(l)} \leq -(m-2) -1 = -(m-1).
$$
The only possibility is that $x_{\sigma(j)} =-(m-1)$ and
$$
\sum_{l=1}^{j-1} x_{\sigma(l)} =1 = x_{\sigma (1)}+x_{\sigma (2)}.
$$
By Lemma \ref{lem:all_distinct}, this implies that we must have $j-1=2$ and thus
$x_{\sigma(j)}=x_{\sigma(3)}=x_i$, a contradiction since by assumption $x_i\neq -(m-1)$.
Assertion \eqref{m-3} is proved.

Since all partial sums in \eqref{m-3} are distinct, included in $\llb -(m-3), m \rrb$ and distinct from
$x_{\sigma (1)}+x_{\sigma (3)}= m+x_i \in \llb 1, m-1 \rrb$, by Lemma \ref{lem:pigeon} \ref{pig2},
we obtain $n \leq 2m-3$, a contradiction.
\end{proof}

Now that we know how negative elements look like, we study the positive ones.

We notice that there must exist in $\ss$ a positive element different from $m$,
otherwise $\ss$ would be of the form $m^u \cdot (-(m-1))^v$ for some
positive integers $u$ and $v$ and, by Lemma \ref{seq} \ref{seq1},
we would get $u=m-1$ and $v=m$ and finally
$$
2m-2 = || \ss || = u+v = (m-1) + m =2m-1,
$$
a contradiction.

Up to a reordering of the elements in the sequence, we may consequently assume that
$$
x_3 \in \llb 1,m-2 \rrb.
$$

\begin{lemma}
\label{pos}
The following holds :
\begin{enumerate}[label={\rm (\roman{*})}]
\item $x_3=1$,
\item if for some $i \in  \llb 1, n \rrb \setminus \{ 3 \}$, $x_i$ is positive, then it is equal to $m$.
\end{enumerate}
\end{lemma}

\begin{proof}
We consider $\sigma$ such that
$$
\sigma (1)=3, \sigma (2)=2, \sigma (3)=1.
$$
The triple $(\ss, 3, \sigma)$ is easily seen to be nyctalopic. We apply Lemma \ref{prolong} to $(\ss, 3, \sigma)$
to extend $\sigma$ in a nyctalopic permutation of  $\llb 1, n \rrb$. We then apply Lemma \ref{prop-nyctalopic}.
We infer that all the partial sums
$\sum_{l=1}^j x_{\sigma(l)}$ belong to $\llb -(m-2), m-1 \rrb$.
Since this set has cardinality $2m-2$, one must have precisely
\begin{equation}
\label{sommepartielle}
\left\{ \sum_{l=1}^j x_{\sigma(l)} \text{ for } j=1,\dots, 2m-3 \right\}  =\llb -(m-2), m-1 \rrb \setminus \{0 \}.
\end{equation}

We consider the function $f$ defined on $\llb 1, n \rrb$ by $f(j)=\sum_{l=1}^j x_{\sigma(l)}$.
One has $f(1)=x_i>0$, $f(2)=x_i+1-m<0$, $f(3)=x_i+1>0$. More generally, if $f(k) >0$, by nyctalopia, one must have
$x_{\sigma(k+1)}<0$ and thus, by Lemma \ref{neg}, $x_{\sigma(k+1)}= -(m-1)$ which implies $f(k+1)=f(k)-(m-1)\leq 0$, where
equality can only happen for $k+1=n$.
Suppose now that the signs of the $f(k)$'s do not alternate when $k \in \llb 1, n-1 \rrb$, then we must have
$$
| \{ 1 \leq k \leq 2m-3 : f(k) <0 \} | > | \{ 1 \leq k \leq 2m-3 : f(k) >0 \} |
$$
which is impossible in view of \eqref{sommepartielle}. Thus the signs alternate and we have
$$
f(1), f(3),\dots, f(2m-3)>0
$$
and
$$
f(2), f(4),\dots, f(2m-4) <0.
$$

We now prove, by induction, that for any integer $j \in \llb 1, m-2 \rrb$, we have $x_{\sigma (2m -2j-1)}=m$.

Indeed, $f(2m-2)=0$ and $f(2m-3)>0$ thus $x_{\sigma (2m-2)}<0$ and thus, by Lemma \ref{neg}, $x_{\sigma (2m-2)}=-(m-1)$.
It follows $f(2m-3)=m-1$. But by the alternance of signs, $f(2m-4)<0$ which implies $x_{\sigma (2m-3)}=f(2m-3)-f(2m-4) \geq m$
and therefore $x_{\sigma (2m-3)}=m$. This proves the statement for $j=1$.

Assume now that for some integer $k \in \llb 1, m-1 \rrb$, the statement is proved for any $j \in \llb 1, k \rrb$.
It follows immediately that
$$
f(2m-3)=m-1, f(2m-5)=m-2,\dots, f(2m-2k-1)=m-k
$$
and
$$
f(2m-4)=-1, f(2m-6)=-2,\dots, f(2m-2k-2)=-k.
$$
Since $f(2m-2k-2)=-k <0$, we must have by the alternance of signs, $f(2m-2k-3)>0$ which implies first that
$x_{\sigma ( 2m-2k-2)}<0$ and thus $x_{\sigma ( 2m-2k-2)}=-(m-1)$. Finally we find that
$$
f(2m-2k-3)=f(2m-2k-2) - x_{\sigma ( 2m-2k-2)} = -k +(m-1)=m-(k+1).
$$
Since, again by the alternance of signs, $f(2m-2k-4)<0$, we have $ x_{\sigma ( 2m-2k-3)}>0$ and one must have
$$
f(2m-2k-4)=f(2m-2k-3) - x_{\sigma ( 2m-2k-3)} \geq (m-(k+1))-m = -(k+1).
$$
Since, by Lemma \ref{lem:all_distinct}, $f$ never takes twice the same value, $f(2m-2k-4) \neq
f(2m-4), f(2m-6),\dots, f(2m-2k-2)$ that is, $f(2m-2k-4) \neq -1,-2,\dots, -k$. This implies finally that
$f(2m-2k-4)= -(k+1)$ and thus that $x_{\sigma ( 2m-2k-3)}=m$, as required to conclude the induction.

Using the statement just proved and the explicit description of the first values of $\sigma$ we obtain the
conclusion of the statement (ii) of the lemma.

By summing all the elements in the zero-sum sequence $\ss$, we obtain thanks to the descriptive lemma
\ref{neg} and what we just proved
$$
0 = x_3 + (m-2)m+(m-1)(-m+1) =x_3-1
$$
thus $x_3=1$ and (i) is proved.
\end{proof}

We are now ready to conclude the proof of Theorem \ref{th:inverse_2m-2}.
One checks the minimality of the sequence $\ss = m^{m-2} \cdot (-(m-1))^{m-1} \cdot 1$, by noticing that,
would this be false, $\AAc ( \{-(m-1), m \} )$ would contain a subsequence of $\ss$, which cannot be
by Lemma \ref{seq} \ref{seq1}.
\end{proof}

\section{Proofs of Theorems \ref{barbie} and \ref{th:main_3}}
\label{avtder}

In this section, we present the proofs of Theorems \ref{barbie} and \ref{th:main_3}.
To ease the reading, these proofs are decomposed into elementary bricks.
Subsection \ref{aa11} contains the proof of all the upper bounds, in particular
the full proof of Theorem \ref{barbie}, its application to Theorem  \ref{th:main_3} (iii)
and the special improvement given in Theorem \ref{th:main_3} (ii).
Subsection \ref{aa22} contains the proof of Theorem \ref{th:main_3} (i).
Finally, Subsection \ref{subsect1} contains the proof of the general lower bounds of Theorem \ref{th:main_3} (ii) and (iii)
(case $m \geq 2$) and Subsection \ref{aa33}  contains the special case $m=1$ in (iii) of Theorem \ref{th:main_3}.

\subsection{Proof of Theorem \ref{barbie} and of the upper bounds in Theorem \ref{th:main_3}}
\label{aa11}

We start from an old question of Riemann and L\' evy.
This was investigated by L\' evy \cite{Levy} himself more than a century ago
but it was Steinitz \cite{Stein13} who gave the first complete proof of the following result.

\begin{theorem}
Let $d$ be a positive integer and $U \subseteq \RR^d$ such that $0 \in U$.
There exists a constant $c$ such that whenever $u_1, \ldots, u_n \in U$ and $u_1 + \cdots + u_n = 0$,
there is a permutation $\pi$ of $\llb 1, n \rrb$ such that $u_{\pi(1)} + \cdots + u_{\pi(i)} \in c \cdot U$
for each $i \in \llb 1,n \rrb$.
\end{theorem}

In this statement, we used the notation $\alpha \cdot U$  for the $\alpha$-dilate of $U$, namely
$$
\alpha \cdot U = \{\alpha u: u \in U\}.
$$

We shall call {\em the Steinitz constant of $U$} the infimum of all constants
$c \in \mathbb \RR^+$ that can be taken in the Theorem.
Steinitz' original results on this constant were later improved by various authors, especially
in the case when $U$ is the closed unit ball relative to a norm $\|\cdot\|$ on $\RR^d$.
In particular, if we consider the superior norm $\|\cdot\|_{\infty}$,
$$
\|(x_1, \ldots, x_d)\|_{\infty} = \max_{1 \le i \le d} |x_i|,
$$
then we denote the corresponding constant by $C_d$: it corresponds to the Steinitz constant of the hypercube.
It is known \cite{Banasz87} (see Remark 3 there) that one has
\begin{equation}
\label{stein}
C_d \leq d+ \frac1d - 1.
\end{equation}

Upper estimates of $C_d$ are immediately made into upper bounds on the Davenport constant.
This is the content of Theorem \ref{barbie}, that we prove now.

\begin{proof}[Proof of Theorem \ref{barbie}]
Consider a sequence $\ss \in \BBc (X)$ and write $\ss = u_1 \cdots u_n$, let $u_i = (u_{i,1}, \ldots, u_{i,d})$ and put
$$
v_i = \left( \frac{u_{i,1}}{m_1}, \ldots, \frac{u_{i,d}}{m_d} \right)
$$
for each $1 \leq i \leq n$, so that  $\| v_i\|_{\infty} \leq 1$ and $v_1 + \cdots + v_n = 0$.
It follows that there exists a permutation $\pi$ of $\llb 1, n \rrb$ such that $v_{\pi(1)} + \cdots + v_{\pi(i)}$
belongs to the box $C_d \cdot B$  where $B$ is the unit ball for $\|\cdot\|_{\infty}$, that is, the hypercube $[-1,1]^d$.
This implies that all the sums $u_{\pi(1)} + \cdots + u_{\pi(i)}$ are lattice points of $C_d \cdot X$.
But the total number of lattice points in $C_d \cdot X = [-C_d m_1, C_d m_1] \times \cdots \times [-C_d m_d, C_d m_d]$
is equal to
$$
\prod_{i=1}^d (2 [C_d m_i ]+1) \leq \prod_{i=1}^d ( 2 C_d m_i +1)
$$
which finally yields, together with Lemma \ref{lem:pigeon} (i), that
$$
\DD(X) \le  \prod_{i=1}^d ( 2C_d m_i +1) .
$$
\end{proof}

The general upper bound of Theorem \ref{th:main_3}
(the one valid for any integral $d \geq 3$) follows immediately from this lemma
applied to $m_1=\cdots=m_d$ and \eqref{stein}.

To prove the particular case $d=2$ (the upper bound in Theorem \ref{th:main_3} (ii)),
we slightly refine this reasoning using a result from \cite{Banasz91}
valid in 2-dimensional spaces, which is a variation on Steinitz' theme.
The main theorem of Banaszczyk's paper \cite{Banasz91} asserts that if $a$ and $b$ are two real numbers
satisfying $a, b \geq 1$ and $a+b \geq 3$, then the following holds:
if $u_1, \ldots, u_n \in B$ ($B$ is again the unit ball relative to the superior norm) and $u_1 + \cdots + u_n = 0$,
there is a permutation $\pi$ of $\llb 1, n \rrb$ such that $u_{\pi(1)} + \cdots + u_{\pi(i)} \in [-a,a]\times[-b,b]$.
Following the same lines as in the preceding proof, this implies, choosing $a=1, b=2$ in this result
that starting from a sequence in $\llb -m_1, m_1 \rrb \times \llb -m_2, m_2 \rrb$,
we may reorder the elements so that the partial sums stay in the rectangle
$\llb -m_1, m_1 \rrb \times  \llb -2m_2, 2m_2 \rrb$. As above, it follows
$$
\DD( \llb -m_1, m_1 \rrb \times   \llb -m_2, m_2 \rrb ) \le  ( 2m_1 +1)(4m_2+1),
$$
which concludes the proof of the result.

\subsection{Proof of Theorem \ref{th:main_3} (i): the case \texorpdfstring{$d=2$}{d=2}, \texorpdfstring{$m=1$}{m=1}}
\label{aa22}

This subsection is devoted to the proof that $ \DD(\llb -1,1\rrb^2)=4$.

We look at the sequence $\tt = (1,-1)\cdot (1, 1) \cdot (-1, 0)^2$. It is easily seen
that $\tt$ is a minimal zero-sum sequence.

Suppose we want to construct a minimal zero-sum sequence of size $n >2$ as long
as possible, then such a sequence $\ss$ can contain at most four distinct elements
(by Lemma \ref{lem:trivial} \ref{lem:trivial:item_ii} and \ref{lem:trivial:item_iii}, $(0,0)$
is not in the sequence and there is at most one point on each line containing $(0,0)$),
and in particular two among $(1,0), (0,1), (-1,0), (0,-1)$, without loss of generality, $(1,0)$ and $(0,1)$.
The point $(-1, -1)$ must be in $\ss$, otherwise the two other points are $(1,1)$ and
up to a symmetry $(1,-1)$, say, but then all the four points have a non negative first coordinate,
leading to a contradiction. Thus  $(-1, -1)$ is in $\ss$. We finally choose as
the fourth point, again without loss of generality by symmetry, $(1,-1)$.
Write $\ss = (1,0)^a \cdot (0,1)^b \cdot (-1,-1)^c \cdot (1,-1)^d$
where $a,b,c$ and $d$ are non-negative integers.
This sequence has sum zero if and only if $a-c+d=0$ and $b-c-d=0$, thus $\ss$ is of the form
$(1,0)^{c-d} \cdot (0,1)^{c+d} \cdot (-1,-1)^{c} \cdot (1,-1)^{d}$ and $c \geq d$.

If $c>d$ then, in particular, $c>0$, which implies that $(1,0)\cdot (0,1) \cdot (-1,-1)$
is a zero-sum subsequence of $\ss$, which implies, by minimality of $\ss$, that
$\ss=(1,0)\cdot (0,1) \cdot (-1,-1)$ and $n=3$. If $c=d$, then $\ss = (0,1)^{2c} \cdot (-1,-1)^{c} \cdot (1,-1)^{c}
=((0,1)^2 \cdot (-1,-1) \cdot (1,-1))^{c}$ and the minimality of $\ss$ implies $c=1$ and $n =4$.
The result is proved.

\subsection{Proof of Theorem \ref{th:main_3} (ii) and (iii): the lower bound in the case \texorpdfstring{$m \geq 2$}{m > 1}}
\label{subsect1}

In all this subsection $m$ is a fixed integer satisfying $m \geq 2$.

We consider the following sequence of zero-sum sequences defined inductively. We let
$$
\ss_1 = m^{m-1} \cdot (-(m-1))^m.
$$
By Corollary \ref{cor:3}, $\ss_1$ belongs to $\AAc ( \llb -m,m \rrb )$ and it has length $|| \ss_1 ||= 2m-1$.
Suppose we already defined a minimal zero-sum sequence $\ss_d$ of $\BBc ( \llb -m,m \rrb ^d)$
of size $|| \ss_d ||= (2m-1)^d$.
Write $\ss_d = x_1 \cdot x_2 \cdots x_n$ where $n= (2m-1)^d$. We define the sequence $\ss_{d+1}$ as follows
\begin{equation}
\label{sd}
\ss_{d+1} = (x_1, m)^{m-1} \cdot (x_2,m)^{m-1} \cdots (x_n,m)^{m-1} \cdot (0, -(m-1))^{mn}.
\end{equation}
It is immediate that $\ss_{d+1} \in \BBc (  \llb -m,m \rrb^{d+1})$ and
$$
|| \ss_{d+1} || = n(m-1) + mn = (2m-1) || \ss_d || = (2m-1)^{d+1}.
$$
This inductive argument implies that, for any positive integer $d$, one has
\begin{equation}
\label{taille}
|| \ss_d ||= (2m-1)^d.
\end{equation}

We start with a basic property of this sequence which will be used in Section \ref{der}.

\begin{lemma}
\label{multiplicity}
For any positive integer $d$, the sequence $\ss_d$ can be written
$$
\ss_d = u_1^{\alpha_1} \cdot  u_2^{\alpha_2} \cdots  u_{d+1}^{\alpha_{d+1}}
$$
where the $u_j$ ($1 \leq j \leq d+1$) are distinct elements of $\llb -m,m \rrb^d$,
the $\alpha_j$ ($1 \leq j \leq d+1$) are positive integers and
$$
\gcd ( \alpha_1, \alpha_2,\dots, \alpha_{d+1})=1.
$$
\end{lemma}

\begin{proof}
The proof is again by induction. For $d=1$, we have $\ss_1 = m^{m-1} \cdot (-(m-1))^m$
and we observe that $\ss_1$ contains two distinct elements repeated $\alpha_1 = m$
and $\alpha_2 = m-1$ times respectively. It is immediate that $\gcd ( m,m-1)=1$ and the
result is proved.

Suppose the result is proved for some integer $d \geq 1$ that is, that
$\ss_d= u_1^{\beta_1} \cdot  u_2^{\beta_2} \cdots  u_{d+1}^{\beta_{d+1}}$
for some distinct elements $u_j$ of $\llb -m,m \rrb^d$ and some positive
integers $\beta_j$ (for $1 \leq j \leq d+1$). A look at \eqref{sd},
taking into account \eqref{taille}, shows immediately that
$$
\ss_{d+1} = (u_1, m)^{(m-1)\beta_1} \cdot (u_2,m)^{(m-1)\beta_2} \cdots (u_{d+1},m)^{(m-1)\beta_{d+1}} \cdot (0, -(m-1))^{m(2m-1)^d}
$$
and we observe that, writing $(u_j,m) = v_j$ for $1 \leq j \leq d+1$ and $v_{d+2}=(0, -(m-1))$,
the $v_j$'s are distinct. Moreover, writing $\alpha_j = (m-1) \beta_j$ for $1 \leq j \leq d+1$ and
$\alpha_{d+2}=m(2m-1)^d$, one obtains
$$
\ss_{d+1} = v_1^{\alpha_1} \cdot  v_2^{\alpha_2} \cdots  v_{d+2}^{\alpha_{d+2}}.
$$
But,
\begin{eqnarray*}
\gcd ( \alpha_1, \alpha_2,\dots, \alpha_{d+2}) & = & \gcd ((m-1)\beta_1, (m-1)\beta_2,\dots, (m-1)\beta_{d+1}, m(2m-1)^d)\\
													& = & \gcd ( \beta_1, \beta_2,\dots, \beta_{d+1}, m(2m-1)^d)
\end{eqnarray*}
since $\gcd (m-1, m(2m-1)^d)=1$. But using the induction hypothesis, we have
$$
\gcd ( \beta_1, \beta_2,\dots, \beta_{d+1}, m(2m-1)^d)  \mid \gcd ( \beta_1, \beta_2,\dots, \beta_{d+1})=1
$$
and finally $\gcd ( \alpha_1, \alpha_2,\dots, \alpha_{d+2}) =1$. The result is proved.
\end{proof}

The following lemma is central for our purpose.

\begin{lemma}
\label{ssdj}
For any two integers $d, u \geq 1$, the non-empty zero-sum subsequences of $\ss_d^{u}$ are
exactly the sequences $\ss_d^j$ for $1 \leq j \leq u$.
\end{lemma}

\begin{proof}
Again, this result is proved by induction. For $d=1$, we consider
$$
\ss_1^u = m^{(m-1)u} \cdot (-(m-1))^{mu}    \in \BBc ( \{ -(m-1), m \} ).
$$
Thus any subsequence $\tt$ of $\ss_1$ must belong to $\BBc ( \{ -(m-1), m \} )$ and,
in view of Lemma \ref{seq} \ref{seq2}, has to be of the form $\tt = \ss_1^j$ for some non-negative integer $j$.

Assume the result is true for some integer $d \geq 1$ and let $\tt$ be a zero-sum subsequence of
$$
\ss_{d+1}^u=(x_1, m)^{(m-1)u} \cdot (x_2,m)^{(m-1)u} \cdots (x_n,m)^{(m-1)u} \cdot (0, -(m-1))^{mnu}
$$
if we denote $\ss_d = x_1 \cdot x_2 \cdots x_n$.
By considering the sequence obtained from $\tt$ by projection on the first $d$ coordinates, which is
nothing but the sequence $\ss_d^{(m-1) u}$ (up to the zeroes obtained from the projection of the elements
$(0, -(m-1))^{mnu}$), and applying the induction hypothesis, we get that
$\tt$ must contain each element $(x_i, m)$ the same number of times, say $j$. It follows that $\tt$
is of the form
$$
\tt =(x_1, m)^{k} \cdot (x_2,m)^{k} \cdots (x_n,m)^{k} \cdot (0, -(m-1))^{l}
$$
for some positive integers $k$ and $l$.
Summing on the last coordinate yields $knm=l(m-1)$. But, by \eqref{taille}, $n= || \ss_k ||=(2m-1)^d$, which gives
$$
k(2m-1)^d m=l (m-1)
$$
from which it follows that $m-1$ divides $k$ in view of $\gcd (m-1,m) = \gcd (m-1, 2m-1)=1$. It follows
$$
k = j (m-1)\quad \text{ and }\quad l=j (2m-1)^d m=jnm
$$
for some integer $j \geq 1$. In other words, $\tt = \ss_{d+1}^j$, which was to be proved to complete the induction step.

The lemma is proved.
\end{proof}

Applying the preceding lemma in the special case $u=1$, we obtain the following result.

\begin{corollary}
For any integer $d \geq 1$, the sequence $\ss_d$ is a minimal zero-sum sequence of $\llb -m,m \rrb^d$.
\end{corollary}

The lower bounds in Theorem \ref{th:main_3} (ii) and (iii) (case $m \geq 2$)
now follow from this Corollary and \eqref{taille}.

\subsection{Proof of Theorem \ref{th:main_3} (iii): the lower bound in the case \texorpdfstring{$m=1$}{m = 1}}
\label{aa33}

If $m=1$, the construction will be slightly different but of the same type.
We could have adapted the argument of the preceding subsection.
However, we can be more direct since an explicit description of $\ss_d$ is possible.

We define the $d+1$ elements of $ \llb -1, 1 \rrb^d$
$$
e_1 = (1,1,\dots, 1),\quad e_2 =( -1,1,\dots, 1),\quad e_3 = ( 0,-1,1,\dots, 1),
$$
$$
 \dots, e_d = ( 0,\dots, 0,-1,1),\quad  e_{d+1} = ( 0,\dots, 0,-1).
$$
In other words, for $1\leq k \leq d+1$, the vector $e_k$ has its $\min (k-2,0)$
first coordinates equal to 0, its $\min(k-1,0)$-th equal to $-1$ and
its coordinates from the $k$-th to the $d+1$-th equal to 1.
We consider the sequence
$$
\ss_d = e_1 \cdot e_2 \cdot e_3^2 \cdot e_4^4\cdots e_d^{2^{d-2}} \cdot e_{d+1}^{2^{d-1}}
$$
so that $\ss_d \in \BBc ( \llb -1, 1 \rrb^d )$ and $||\ss_d|| = 2^d$.

It remains to prove that this sequence is minimal. Consider $\tt$ a non-empty zero-sum subsequence of $\ss_d$.
Let $j$ be the minimal index ($1 \leq j \leq d+1$) such that there is at least one element
in the sequence having a non-zero $j$-th coordinate. If $j>1$, then any element in $\tt$ is
one of the $e_k$'s for $k \geq j+1$ but then all the elements of the sequence have
a nonpositive $j$-th coordinate, and at least one has a strictly negative one. Thus $\tt$
cannot be a zero-sum sequence. It follows $j=1$ and $\tt$ must contain either $e_1$ and $e_2$,
and thus both, looking at the first coordinate.

We now prove by induction that, for $k\geq 2$, $\tt$ must contain each $e_k$
with multiplicity $2^{k-2}$. We just proved it for $k=2$. Suppose this is true for some value
of $k<d+1$, then considering the $k+1$-th coordinate of the sum of $\tt$, we obtain that
the multiplicity of $e_{k+1}$ must be equal to
$$
1+1+2+\cdots +2^{k-2}=2^{k-1}.
$$
This completes the induction step and finally the proof that $\tt=\ss_d$.

Thus $\ss_d$ is minimal and, since $||\ss_d|| = 2^d$, the lower bound of Theorem \ref{th:main_3} (iii)
is proved for $m=1$.

\section{Proofs of Theorem \ref{th:main_4} and Theorem \ref{th5}}
\label{der}

We start with the proof of the Theorem \ref{th:main_4}.

\begin{proof}[Proof of Theorem \ref{th:main_4}]
Take a sequence $\ss \in \BBc ( G \times X )$ of length larger than or equal to $\DD(G) \DD (X) +1$.
Since this is larger than $\DD(X)$ we may extract from this sequence
a subsequence $\ss_1$ which sums minimally to zero on the second component.
By definition of an element of $\AAc (X)$, this has a length at most $\DD(X)$.
Removing this subsequence from $\ss$, we get a new sequence $\ss \cdot \ss_1^{-1}$ (we denote in this way
the sequence obtained from $\ss$ after deleting from it the subsequence $\ss_1$) and we have
$$
|| \ss \cdot \ss_1^{-1} || \geq \DD(G) \DD (X) +1 - \DD (X) =  (\DD(G) -1) \DD (X) +1.
$$
While $\ss \cdot \ss_1^{-1}$ does not a priori belong to $\BBc (G \times X)$ ($\ss_1$ may have a
non-zero sum on its first component), this sequence sums to zero on the second component.
We can therefore continue this process and build recursively the sequences $\ss_2,\dots, \ss_l$ such that their projection
on the second component belongs to $\AAc ( X )$. Since $|| \ss_j || \leq \DD(X)$ for
each index $j \geq 1$, the process can continue as long as $l \leq \DD(G)$.
Thus, we can assume that we have built $l=\DD(G)$ distinct subsequences of $\ss$, namely
$\ss_1, \ss_2, \dots, \ss_l$, each summing to zero on the second component.
For each $j \in \llb 1, l \rrb$, we call $g_j \in G$ the sum of the sequence $\ss_j$
on the first component. Notice that  $\ss \cdot \ss_1^{-1}\ss_2^{-1} \dots \ss_l^{-1}$ is non-empty since
$$
|| \ss \cdot \ss_1^{-1}\ss_2^{-1} \dots \ss_l^{-1} || = || \ss || - ( || \ss_1 || + || \ss_2 || + \cdots + || \ss_l|| )
\geq \DD(G) \DD (X) +1 -l \DD (X)=1.
$$
Applying the definition of the Davenport constant of $G$ to the sequence $\tt = g_1 \cdot g_2 \cdots g_l$
(notice that it is a priori not a zero-sum sequence in $G$), we can extract from $\tt$
a subsequence $g_{i_1} \cdot g_{i_2} \cdots g_{i_q}$, for some $q \leq l$ and indices
$1 \leq i_1 < i_2 < \cdots < i_q \leq l$, which sums to $0$ in $G$.
Finally, we consider the subsequence of $\ss$ defined as $\ss' = \ss_{i_1} \cdot \ss_{i_2} \cdots \ss_{i_q}$.
It is a proper subsequence of $\ss$ and we check immediately that
$$
\sum_{x \in \ss'} x = \sum_{j=1}^q \sum_{x \in \ss_{i_j}} x = \sum_{j=1}^q (g_{i_j},0)=0,
$$
which proves that $\ss$ cannot be minimal and, consequently, that $\DD ( G \times X ) \leq  \DD ( G )\ \DD (X )$.
\end{proof}

We finally prove Theorem \ref{th5}.

\begin{proof}[Proof of Theorem \ref{th5}]
Let $n= | G |$ and $g$ be a generator of $G$.

If $m \geq 2$, we use the sequence $\ss_d$ introduced in Section \ref{avtder} (Subsection \ref{subsect1}).
In view of Lemma \ref{multiplicity}, we can write it in the form
$$
\ss_d = u_1^{\alpha_1} \cdot  u_2^{\alpha_2} \cdots  u_{d+1}^{\alpha_{d+1}}
$$
with distinct elements $u_j \in \llb -m,m \rrb^d$ and positive integers $\alpha_j$ (for $1 \leq j \leq d+1$).
We also have
$$
\gcd ( \alpha_1, \alpha_2,\dots, \alpha_{d+1})=1
$$
which implies by B\' ezout's theorem, that we can find integers $w_1,w_2,\dots, w_{d+1}$ such that
\begin{equation}
\label{grrrr}
\alpha_1 w_1 +  \alpha_2 w_2 + \cdots + \alpha_{d+1} w_{d+1} = 1.
\end{equation}
We finally define the sequence
$$
\tt = (w_1 g, u_1)^{n\alpha_1} \cdot (w_2 g, u_2)^{n\alpha_2} \cdots (w_{d+1} g, u_{d+1})^{n\alpha_{d+1}}.
$$
By \eqref{grrrr} and $\ss_d \in \BBc ( \llb -m,m \rrb^d )$, it is immediate to check that
$$
\sum_{x \in \tt} x = \sum_{j=1}^{d+1} \alpha_j n (w_j g, u_j) = \left( ng, n  \sum_{j=1}^{d+1} \alpha_j u_j \right) =(0,0).
$$
Thus $\tt \in \BBc ( G \times \llb -m,m \rrb^d )$.

Let us show that $\tt$ is minimal. Select a non-empty zero-sum subsequence of it, say $\uu$.
By Lemma \ref{ssdj} applied to the second component, which is nothing but $\ss_d^n$, we observe that $\uu$ must be of the form
$$
\uu = (w_1 g, u_1)^{q\alpha_1 } \cdot (w_2 g, u_2)^{q\alpha_2} \cdots (w_{d+1} g, u_{d+1})^{q\alpha_{d+1}}
$$
for some positive integer $q \leq n$.
By summing $\uu$, we get, again by \eqref{grrrr} and $\ss_d \in \BBc ( \llb -m,m \rrb^d )$,
$$
\sum_{x \in \uu} x = \left( q \left(\sum_{j=1}^{d+1} \alpha_j w_j \right) g, q\sum_{j=1}^{d+1} \alpha_j u_j \right)
= (qg,0)
$$
a sum which can be zero only for $q$ a multiple of $|G|=n$, $g$ being a generator. Thus $q=n$. It follows that
$\tt \in \AAc ( G \times  \llb -m,m \rrb^d)$.

The theorem now follows from $|| \tt || = n || \ss_d ||= (2m-1)^d \DD (G)$.

If $m=1$, the same proof applies in an analogous way. This is even simpler since we can take
all but one (namely, $w_1$) of the $w_j$'s equal to zero in view of $\alpha_1=1$.
\end{proof}

\medskip

\section*{Acknowledgments}
The authors are indebted to Alfred Geroldinger for having raised a question that motivated this paper.
The second author is also thankful to Eric Balandraud and Benjamin Girard for useful discussions.

\end{document}